\documentclass[11pt,draft]{amsart}

\usepackage[T1]{fontenc}
\usepackage{amsfonts, amsthm, epsfig, amsmath, amssymb, color,eucal,dsfont,verbatim}
\usepackage{textcomp}

\usepackage[english]{babel}

\usepackage{hyperref}

\usepackage{color}

\newcommand{\normal}{\color{black}}

\theoremstyle{plain}
\newtheorem{thm}{Theorem}[section]
\newtheorem{prop}{Proposition}[section]
\newtheorem{cor}{Corollary}[section]

\theoremstyle{definition}

\theoremstyle{remark}
\newtheorem{rem}{Remark}[section]

\renewcommand{\star}{\circledast}

\newcommand{\rdd}{{\mathbb{R}^d}}
\newcommand{\I}{\mathds{1}}
\renewcommand{\Re}{\mathbb{R}}
\newcommand{\Ff}{\mathcal{F}}

\newcommand{\prt}{\partial}

\newcommand{\be}{\begin{equation}}
\newcommand{\ee}{\end{equation}}
\newcommand{\ba}{\begin{aligned}}
\newcommand{\ea}{\end{aligned}}
\newcommand{\real}{\mathds{R}}
\newcommand{\Ee}{\mathds{E}}

\begin{document}

\title{Accuracy of discrete approximation for integral functionals of Markov processes}

\author{Iu. V. Ganychenko}
\address{Department of Probability Theory, Statistics and Actuarial Mathematics, Kyiv National Taras Shevchenko University, Kyiv 01601, Ukraine}
\email{iurii\_ganychenko@ukr.net}
\author{V. P. Knopova}
\address{V.M.Glushkov Institute of Cybernetics NAS of Ukraine, Acad. 40, Glushkov Ave.,  Kyiv 03187, Ukraine}
\email{vicknopova@googlemail.com}
\author{A. M. Kulik}
\address{Institute of Mathematics, Ukrainian National Academy of Sciences, 01601 Tereshchenkivska str. 3, Kyiv, Ukraine}
\email{kulik@imath.kiev.ua}




\maketitle

\begin{abstract}
    \noindent The article is devoted to the estimation of the rate of convergence of integral functionals of a Markov process. Under the assumption that the given Markov process admits a transition probability density which is differentiable in $t$ and the derivative has an integrable upper bound of a certain type, we derive the accuracy rates for strong and weak   approximations of the functionals by  Riemannian sums. Some examples are provided.

\end{abstract}


\section{Introduction} Let $X_t$, $t\geq 0$, be a  Markov process   with values in  $\mathbb{R}^d$.  Consider an integral functional of the form
\begin{equation}\label{IT}
I_T(h)=\int_0^Th(X_t)\, dt,
\end{equation}
where $h:\Re^d\to \Re$ is a given measurable function. In this paper we investigate the  accuracy of the approximation of $I_T(h)$ by the Riemannian  sums
$$
I_{T,n}(h)={T\over n}\sum_{k=0}^{n-1}h(X_{(kT)/n}),\quad n\geq 1.
$$
The function $h$ is assumed to be bounded, only; i.e.,  we do not impose any regularity assumptions on $h$. In particular, under this  assumption the class of integral functionals  which we investigate contains the class of  \emph{occupation time} type functionals (for which $h=1_A$ for a fixed $A\in \mathcal{B}(\Re^d)$), which are of particular importance.

Integral functionals arise naturally in a wide class  of stochastic representation formulae and applied stochastic models. It is very typical that exact calculation of the respective probabilities and/or expectations is hardly possible, which naturally  suggests the usage of the approximation methods. As an  example of such a situation, we mention the so-called \emph{occupation time option} \cite{Linetsky}, whose price is actually given by an expression similar to the Feynman-Kac formula. The exact calculation of the price is possible only in the particular case when the underlying process is a L\'evy process which is ``spectrally negative'' (i.e. does not have positive jumps, see \cite{Guerin}), and practically more realistic cases of general L\'evy processes, solutions to L\'evy driven SDE's, etc. can be treated  only numerically. To estimate the rate of convergence of the respective Monte-Carlo approximative methods, one needs to estimate the accuracy of various approximation steps involved in the algorithm. In this paper we  focus
on solving of such a problem for the discrete approximation of the integral functional  of type \eqref{IT}.

For diffusion processes, this problem was studied in \cite{Gobet} and recently in \cite{Kohatsu-Higa}, by means of the methods involving the particular structural features of the process, e.g. the Malliavin calculus tools. On the other hand, in two recent papers \cite{kul-gan}, \cite{kul-gan2},  an alternative method is developed, which exploits only the basic Markov structure of the process and  the additive structure of the integral functional and its discrete approximations.  One  of the aims of this paper is to extend this method for a wider class of Markov processes. To explain our goal more  in details, let us formulate our principal assumption on the process $X$.

\begin{itemize}
  \item[\textbf{X.}]  The transition probability $P_t(x,dy)$ of $X$ admits a density $p_t(x,y)$ w.r.t. the Lebesgue measure on $\Re^d$. This density  is differentiable w.r.t. $t$, and its derivative possesses the bound
\be\label{der_bound}
\Big|\prt_tp_t(x,y)\Big|\leq B_{T,X}t^{-\beta} q_{t,x}(y), \quad t\leq T,
\ee
for some  $B_{T,X}\geq 1, \beta\geq 1$, and measurable function $q$, such that for each fixed $t,x$ the function $ q_{t,x}(\cdot)$ is a distribution density.
\end{itemize}

In \cite{kul-gan}, \cite{kul-gan2}, the condition similar to \textbf{X} was formulated with $\beta=1$. Such a condition is verified for the particularly important classes of diffusion process and symmetric $\alpha$-stable processes. However, in some natural cases one can expect to get \eqref{der_bound} only with $\beta>1$. As  the simplest and the most illustrative example one can take an $\alpha$-stable process with drift:
$$X_t=ct+Z_t,
$$ where $c\not=0$ and $Z$ is a (e.g. symmetric) $\alpha$-stable process. Then
$$
p_t(x,y)=t^{-d/\alpha}g^{(\alpha)}\left(y-x-ct\over t^{1/\alpha}\right),
$$
where $g^{(\alpha)}$ denotes the distribution density of $Z_1$. Straightforward calculation shows that \eqref{der_bound} holds true with $\beta=\max(1, 1/\alpha)$, which is strictly greater than 1 when $\alpha<1$. Since the L\'evy noises  are now used extensively in various applied models, the  simple calculation made above shows that it is highly desirable to extend the results of \cite{kul-gan} and  \cite{kul-gan2},  which deal with the ``diffusive like'' class of processes satisfying \textbf{X} with $\beta=1$,  to the  more general case of \textbf{X} with arbitrary $\beta\geq 1$.

Another aim  of this paper is to develop  the tools which would allow us to get the bounds of the form \eqref{der_bound} for a  wider class of solutions of L\'evy driven SDEs. One result of such a type is given in the recent preprint \cite{KK14}, with the process $X$ being  a solution of the SDE
\be\label{SDE}
dX_t=b(X_t)\, dt+\sigma(X_{t-})\, dZ_t
\ee
 where $Z$ is  a symmetric  $\alpha$-stable process. The method used therein is a version of the parametrix method, and it is quite sensitive to the form of the L\'evy measure of the process $Z$ on the entire $\Re^d$. Recently, apart from the  stable noises, various types of ``locally stable'' noises are frequently used in applied models: tempered stable processes, damped stable processes, etc. Heuristically,   for a ``locally stable'' process its ``small jumps behavior'' is the same as for the stable one, but the ``large jumps behavior'' of the former is drastically different from ``tail behavior'' of the L\'evy measure. Since   \eqref{der_bound} is genuinely related to ``local behavior'' of the process, one can expect that the results of \cite{KK14} should have a natural extension to the case of  ``locally stable'' $Z$. However, to make  such a conjecture rigorous is a sophisticated problem; the main reason here is that the parametrix method treats the transition probability of a  L\'evy process as the  ``zero approximation'' for the unknown transition probability density $p_t(x,y)$, and hence any bound for $p_t(x,y)$, which one may expect to design within this method, is at least as complicated as respective bound for the process $Z$. On the other hand, there is an extensive  literature on the estimates of transition probability densities for L\'evy processes (e.g. \cite{BGR14},  \cite{KK12a}, \cite{K13}, \cite{KR15}, \cite{M12}, \cite{PT69},  \cite{RS10}, \cite{St10a}, \cite{St10b}, \cite{W07}; this list is far from being complete), which shows that these densities inherit the structure of the densities of the corresponding L\'evy measures. In particular, in order to get the exact two-sided bounds for $p_t(x,y)$ one should impose quite non-trivial structural assumptions on the ``tails'' of the L\'evy measure even in a comparatively simple ``locally stable'' case. Motivated by this observation on one hand, and by the initial approximation problem which suggests the condition  \eqref{der_bound} on the other hand, we  pose the following general question: \emph{Is it possible to give  a ``rough'' upper bound, which would be the same for a large class of  transition probability densities of ``locally stable processes'', without assuming complicated conditions on the ``tails'' of their L\'evy measures?} The answer is positive, and it roughly says that one can get the bound \eqref{der_bound}, where at the left hand side we have  the transition probability density of the SDE driven by a ``locally stable'' process, and   at the right hand side we have a (properly shifted) transition probability density of an $\alpha$-stable process. This bound is not necessarily  precise: the power-type ``tail'' of the $\alpha$-stable density might be essentially larger than the ``tail'' e.g. for \emph{exponentially tempered} $\alpha$-stable law. The gain is, however,  that under a mild set of assumptions we obtain a uniform-in-class upper bound, which is clearly easy to use in applications.   To keep the exposition reasonably compact, we treat this problem in a comparatively simple case of a one-dimensional SDE of the form  \eqref{sde1}, see below. The extension of these results to a  more general multidimensional case is much more technical,  and we postpone it to a separate publication.

The structure of the paper is the following. In Section \ref{s2} we formulate and prove two  our main results concerning the  accuracy of the \emph{strong}  and \emph{weak} approximations of an integral functional by  Riemannian sums, provided that condition \textbf{X} is satisfied.  In Section \ref{s3} we outline a version of the parametrix method, which makes it possible to obtain \eqref{der_bound} for solutions to L\'evy driven SDEs without strong structural assumptions on the ``tails'' of the L\'evy measure of the noise. In Section \ref{s4} an application for the price of an occupation time option is given.

\section{Accuracy of discrete approximation for integral functionals}\label{s2}

In this section we will prove two results. The first one concerns the ``strong approximation rate'', i.e. the control on  the $L_p$-distance between $I_{T}(h)$ and its approximation $I_{T,n}(h)$.

\begin{thm}\label{t1} Suppose that  \textbf{X} holds.  Then for any $p>0$
$$
\left(\Ee_x\Big|I_{T}(h)-I_{T,n}(h)\Big|^p\right)^{1/ p}\leq C_{T,p} \|h\| (D_{T,\beta} (n))^{1/2},
$$
where $\|h\|=\sup_x|h(x)|,$
\begin{equation}\label{DC}
D_{T,\beta} (n) =
\begin{cases}
n^{-1} \log n,& \beta=1,\\
\max\left(1,\frac{T^{1- \beta}}{\beta-1}\right) n^{-1/ \beta},& \beta>1,
\end{cases}\quad C_{T,p}=
\begin{cases}
(14p(p-1)B_{T,X})^{1/2} T,& p\geq 2,\\
C_{T,2}=(28 B_{T,X})^{1/2} T,& p\in (0,2).
\end{cases}
\end{equation}
\end{thm}

\begin{rem}  This theorem extends \cite[Theorem 2.1]{kul-gan}, where it was assumed that $\beta=1$. 
\end{rem}

The second result concerns the ``weak approximation'', i.e. the   control on  the difference between the expectations of certain terms, which involve  $I_{T}(h)$ together with  its approximation $I_{T,n}(h)$.

\begin{thm}\label{t2} Suppose that   \textbf{X} holds. Then for any $k \in \mathbb{N}$ and any bounded function $f$  we have
\begin{equation}\label{t2-eq}
\Big|\Ee_x (I_{T}(h))^k f(X_T)- \Ee_x(I_{T,n}(h))^kf(X_T)\Big|
\leq 2^{\beta\vee 2} k^2 B_{T,X} T^{k+1} \|h\|^k \|f\| D_{T,\beta} (n).
\end{equation}
\end{thm}
\begin{rem}  This theorem extends  \cite[Theorem~1.1]{kul-gan2}, where it was assumed that $\beta=1$. In the proof below, we concentrate on the case $\beta>1$.
\end{rem}

Using the Taylor expansion, one can obtain directly the following corollary of Theorem \ref{t2}.

\begin{cor}\label{cor1}
 Suppose that \textbf{X} holds, and  let $\varphi$ be an  analytic function defined in some neighbourhood of 0. In addition, suppose that the  constants $D_{\varphi},R_{\varphi}>0$ are   such that
$$
\Big|\frac{\varphi^{(m)}(0)}{m!}\Big| \leq D_{\varphi} \left( \frac{1}{R_{\varphi}} \right)^m,\quad m\geq 0.
$$
Then for any bounded function $f$ and a function $h$ such that $T\|h\|< R_\varphi$, we have the  following bound:
$$
\Big|\Ee_x \varphi( I_{T}(h)) f(X_T)- \Ee_x \varphi( I_{T,n}(h))f(X_T)\Big| \leq
C_{T,X,h,\varphi}  \|f\| D_{T,\beta} (n),
$$
where
$$
C_{T,X,h,\varphi}=2^{\beta\vee 2} D_{\varphi}B_{T,X}\frac{ T^2 \|h\|}{R_{\varphi}} \left(1+\frac{ T\|h\|}{R_{\varphi}}\right)\left(1-\frac{ T\|h\|}{R_{\varphi}}\right)^{-3}.
$$
\end{cor}

Before proceeding to the proof of Theorem \ref{t1}, we give an auxiliary result this  proof is based on. This result is, in fact,  a weaker version
of Theorem \ref{t2} with $k=1$ and $f\equiv 1$, but we give it separately to make the exposition more transparent.

\begin{prop}\label{p1} Suppose that  \textbf{X} holds. Then
$$
\Big|\Ee_xI_{T}(h)-\Ee_xI_{T,n}(h)\Big|\leq 5B_{T,X} T \|h\| D_{T,\beta} (n).
$$
\end{prop}
\begin{proof} Let us introduce the notation used throughout the whole section: for $t\in [kT/n, (k+1)T/n), k\geq 0 $, we put
$\eta_n(t)={kT\over n}, \ \zeta_n(t)={(k+1)T\over n}$; that is, $\eta_n(t)$ is the point of the partition $\{Tk/n,\, k\geq 0\}$ of the time axis, closest to $t$ from the left, and $\zeta_n(t)$ is the point closest to $t$ from the right, which is strictly larger than $t$.

We have
\begin{align*}
\Ee_xI_{T}(h)-\Ee_xI_{T,n}(h) &= \int_0^T \Ee_x[h(X_s) - h(X_{\eta_n(s)})]\,ds\\
&
= \int_0^T \int_{\Re^d}h(y)  [p_s(x,y) - p_{\eta_n(s)}(x,y)]\,dyds \\
&= M_1 + M_2,
\end{align*}
where
$$\ba
&M_1 = \int_0^{k_{n,\beta}T/n} \int_{\Re^d}h(y)  [p_s(x,y) - p_{\eta_n(s)}(x,y)]\,dyds,\\&
M_2 = \int_{k_{n,\beta}T/n}^T \int_{\Re^d}h(y)  [p_s(x,y) - p_{\eta_n(s)}(x,y)]\,dyds,
\ea$$
for some  $1\leq k_{n,\beta} \leq n$, which will be chosen  later. We estimate each term separately.

For $M_1$ we have
$$
|M_1| \leq \|h\| \int_0^{k_{n,\beta}T/n} \int_{\Re^d}  [p_s(x,y) + p_{\eta_n(s)}(x,y)]\,dyds = 2\|h\| T \frac{k_{n,\beta}}{n}.
$$
Further,  using (\ref{der_bound}), we  get
\be\label{M2}\ba
|M_2| &\leq \|h\| \int_{k_{n,\beta}T/n}^T \int_{\Re^d}  |p_s(x,y) - p_{\eta_n(s)}(x,y)|\,dyds
\\&
 \leq \|h\| \int_{k_{n,\beta}T/n}^T \int_{\eta_n(s)}^s \int_{\Re^d}  |\partial_u p_u(x,y)|\,dy duds
\\
&\leq B_{T,X}\|h\| \int_{k_{n,\beta}T/n}^T \int_{\eta_n(s)}^s \int_{\Re^d} u^{-\beta} q_{u,x}(y)
 \,dy duds  = B_{T,X} \|h\| \int_{k_{n,\beta}T/n}^T \int_{\eta_n(s)}^s u^{-\beta}\,duds\\
 &= B_{T,X} \|h\| \sum_{i = k_{n,\beta}}^{n-1} \int_{iT/n}^{(i+1)T/n} \int_{iT/n}^s u^{-\beta}\,duds
 \\&= B_{T,X} \|h\| \sum_{i = k_{n,\beta}}^{n-1} \int_{iT/n}^{(i+1)T/n} \int_u^{(i+1)T/n} u^{-\beta}\,dsdu\\&
\leq \frac{T}{n} B_{T,X} \|h\| \sum_{i = k_{n,\beta}}^{n-1} \int_{iT/n}^{(i+1)T/n}  u^{-\beta}\,du = \frac{T}{n} B_{T,X} \|h\| \int_{k_{n,\beta}T/n}^{T}  u^{-\beta}\,du.
\ea\ee

Now we finalize the argument.

1) If  $\beta = 1$,  put $k_{n,\beta} = 1, \ n\geq 1$. Then we get
$$\ba
&|M_1| \leq 2\|h\| T n^{-1},\\&
|M_2| \leq \frac{T}{n} B_{T,X} \|h\| \int_{T/n}^{T}  u^{-1}\,du = B_{T,X} T \|h\| n^{-1} \log n.
\ea$$

2) If $\beta > 1$,  put $k_{n,\beta} = [n^{1-1/\beta}]+1$. Then
$$
|M_1| \leq 2\|h\| T  \frac{[n^{1-1/\beta}]+1}{n} \leq 2\|h\| T  \frac{n^{1-1/\beta}+1}{n} \leq 4\|h\| T  n^{-1/\beta}.
$$
To estimate $M_2$ observe that
\begin{equation}\label{beta1}
\begin{split}
\int_{k_{n,\beta}T/n}^{T}  u^{-\beta}\,du\leq \frac{T^{1-\beta}}{\beta-1}\left(\frac{k_{n,\beta}}{n}\right)^{1-\beta}\leq   \frac{T^{1-\beta}}{\beta-1} \left(\frac{n^{1-1/\beta}}{n}\right)^{1-\beta}\leq \frac{T^{1-\beta}}{\beta-1}n^{1-1/\beta}.
\end{split}
\end{equation}
Therefore,
\begin{align*}
|M_2| &\leq \frac{T}{n} B_{T,X} \|h\| \int_{k_{n,\beta}T/n}^{T}  u^{-\beta}\,du
\leq \frac{B_{T,X}}{\beta-1} T^{2- \beta} \|h\| n^{-1/\beta}.
\end{align*}

\end{proof}

\begin{proof}[Proof of Theorem \ref{t1}] Since we can obtain the required bound for $p<2$ from the bound with $p=2$ by the H\"older inequality, we  consider the case $p\geq 2$ only.

Define
$$
J_{t,n}(h):=I_{t}(h)-I_{t,n}(h)=\int_0^t \Delta_n(s)ds, \quad \Delta_n(s):=h(X_s) - h(X_{\eta_n (s)}).
$$
By definition,  the  function $t\mapsto J_{t,n}(h)$ is absolutely continuous. Then  using the Newton-Leibnitz formula twice we get
$$
\Big|J_{T,n}(h)\Big|^p=p(p-1)\int_0^T\Big|J_{s,n}(h)\Big|^{p-2}\Delta_n(s)\left(\int_s^T\Delta_n(t)\, dt\right)ds.
$$
Therefore,
$$
\Big|J_{T,n}(h)\Big|^p\leq p(p-1)(H_{T,n,p}^1(h)+H_{T,n,p}^2(h)),
$$
where
$$
 H_{T,n,p}^1(h)=\int_0^T\Big|J_{s,n}(h)\Big|^{p-2}|\Delta_n(s)|\left|\int_s^{\zeta_n(s)}\Delta_n(t)\, dt\right|ds,
$$
$$
H_{T,n,p}^2(h)=\int_0^T\Big|J_{s,n}(h)\Big|^{p-2}\Delta_n(s)\left(\int_{\zeta_n(s)}^T\Delta_n(t)\, dt\right)ds.
$$
Let us  estimate separately the expectations of $ H_{T,n,p}^1(h)$ and $H_{T,n,p}^2(h)$. By the H\"older inequality,
\begin{align*}
\Ee_x H_{T,n,p}^1(h)&\leq \left(\Ee_x\int_0^T\Big|J_{s,n}(h)\Big|^{p}\, ds\right)^{1-2/p}\left(\Ee_x\int_0^T|\Delta_n(s)|^{p/2}\left|\int_s^{\zeta_n(s)}\Delta_n(t)\, dt\right|^{p/2}\, ds\right)^{2/p}
\\
&\leq \left(\Ee_x\int_0^T\Big|J_{s,n}(h)\Big|^{p}\, ds\right)^{1-2/p} \left((2\|h\|)^{p/2} T (2\|h\|)^{p/2} \left(\frac{T}{n}\right)^{p/2}\right)^{2/p}\\
&
= 4  T^{1+2/p} n^{-1} \|h\|^2  \left(\Ee_x\int_0^T\Big|J_{s,n}(h)\Big|^{p}\, ds\right)^{1-2/p}.
\end{align*}

Further, observe that for every $s$ the variables $$\Delta_n(s), \quad |J_{s,n}(h)|^{p-2}\Delta_n(s)
$$ are $\Ff_{\zeta_n(s)}$-measurable; here and below $\{\Ff_t, t\geq 0\}$ denotes the natural filtration for the process $X$. Hence,
$$\ba
\Ee_xH_{T,n,p}^2(h)&=\Ee_x\left(\int_0^T\Big|J_{s,n}(h)\Big|^{p-2}\Delta_n(s)\Ee_x\left(\int_{\zeta_n(s)}^T\Delta_n(t)\, dt\Big| \Ff_{\zeta_n(s)}\right)ds\right)\\&
\leq \Ee_x\left(\int_0^T\Big|J_{s,n}(h)\Big|^{p-2}|\Delta_n(s)|\left|\Ee_x\left(\int_{\zeta_n(s)}^T\Delta_n(t)\, dt\Big| \Ff_{\zeta_n(s)}\right)\right|ds\right).
\ea
$$
By Proposition \ref{p1} and the Markov property of $X$, we have
$$
\left|\Ee_x\left(\int_{\zeta_n(s)}^T\Delta_n(t)\, dt\Big|\Ff_{\zeta_n(s)}\right)\right| = \left|E_{X_{\zeta_n(s)}}\int_{0}^{T-\zeta_n(s)}\Delta_n(t)\, dt\right| \leq 5B_{T,X} T D_{T,\beta} (n) \|h\|.
$$
Therefore, using the H\"older inequality, we get
$$\ba
\Ee_xH_{T,n,p}^2(h)&\leq 5B_{T,X} T D_{T,\beta} (n) \|h\| \left(\Ee_x\int_0^T\Big|J_{s,n}(h)\Big|^{p}\, ds\right)^{1-2/p}\left(\Ee_x\int_0^T|\Delta_n(s)|^{p/2}\, ds\right)^{2/p}\\&
\leq 10B_{T,X} T^{1+2/p} D_{T,\beta} (n) \|h\|^2\left(\Ee_x\int_0^T\Big|J_{s,n}(h)\Big|^{p}\, ds\right)^{1-2/p}.
\ea
$$
Note that $n^{-1}\leq D_{T,\beta} (n)$, hence the above bounds for $\Ee_x H_{T,n,p}^1(h)$ and $\Ee_xH_{T,n,p}^2(h)$ finally yield the estimate \be\label{recur_bound}
\Ee_x \Big|J_{T,n}(h)\Big|^{p}\leq 14p(p-1)B_{T,X} T^{1+2/p} D_{T,\beta} (n) \|h\|^2\left(\Ee_x\int_0^T\Big|J_{s,n}(h)\Big|^{p}\, ds\right)^{1-2/p}.
\ee

It can be  easily  seen that the above inequality also holds true if $J_{T,n}(h)$ in the left hand side is replaced by $J_{t,n}(h)$.  Taking the integral over $t\in [0,T]$, we get
$$
\Ee_x\int_0^T\Big|J_{t,n}(h)\Big|^{p}\, dt\leq 14p(p-1)B_{T,X} T^{2+2/p} D_{T,\beta} (n) \|h\|^2\left(\Ee_x\int_0^T\Big|J_{s,n}(h)\Big|^{p}\, ds\right)^{1-2/p}.
$$

Because $h$ is bounded, the left hand side expression in the above inequality is finite. Hence, resolving this inequality, we get
$$
\Ee_x\int_0^T\Big|J_{s,n}(h)\Big|^{p}\,ds\leq (14p(p-1)B_{T,X})^{p/2} T^{p+1} (D_{T,\beta} (n))^{p/2} \|h\|^p,
$$
which together with (\ref{recur_bound}) gives the required statement.\normal

\end{proof}

\begin{proof}[Proof of Theorem \ref{t2}] Denote
$$S_{k,a,b}:= \{(s_1,s_2,...,s_k) \in \mathbb{R}^k : a\leq s_1 < s_2 < ...< s_k \leq b\}, \ k \in \mathbb{N}, \  a,b \in \mathbb{R}.$$

We have
\be\ba \label{sumJi}
&\Ee_x \Big[(I_{T}(h))^k-(I_{T,n}(h))^k \Big] f(X_T)\\&
= k!\,\Ee_x \int_{S_{k,0,T}}[h(X_{s_1})h(X_{s_2})...h(X_{s_k}) - h(X_{\eta_n(s_1)})h(X_{\eta_n(s_2)})...h(X_{\eta_n(s_k)})] f(X_T) \prod_{i=1}^k ds_{i} \\&
= k!\, \int_{S_{k,0,T}}\int_{(\mathbb{R}^d)^{k+1}}\left(\prod_{i=1}^kh(y_i)\right)f(z)  \left(\prod_{i=1}^k p_{s_i-s_{i-1}}(y_{i-1},y_i)\right)p_{T-s_k}(y_k,z)dz \prod_{j=1}^k dy_j \prod_{i=1}^k ds_{i}\\&
-k!\, \int_{S_{k,0,T}}\int_{(\mathbb{R}^d)^{k+1}}\left(\prod_{i=1}^kh(y_i)\right)f(z) \left(\prod_{i=1}^k p_{\eta_n(s_i)-\eta_n(s_{i-1})}(y_{i-1},y_i)\right)\\&
\times p_{T-\eta_n(s_k)}(y_k,z) dz \prod_{j=1}^k dy_j \prod_{i=1}^k ds_{i}\\&
= k!\sum_{r=1}^{k}\int_{S_{k,0,T}}\int_{(\mathbb{R}^d)^{k+1}} \left(\prod_{i=1}^kh(y_i)\right)f(z) J_{s_1, \dots, s_k,T}^{(r)}(x,y_1, \dots, y_k,z)dz \prod_{j=1}^k dy_j \prod_{i=1}^k ds_{i},
\ea\ee
where the convention  $s_0 = 0, s_{k+1}=T, y_0 = x, y_{k+1}=z$ is used and the functions $J^{(r)}, r=1, \dots, k$ are defined by the relations
\begin{align*}
J_{s_1, \dots, s_k,T}^{(r)}&(x,y_1, \dots, y_k,z) = \left(\prod_{i=1}^{r-1} p_{\eta_n(s_i)-\eta_n(s_{i-1})}(y_{i-1},y_i)\right)\\
&\quad \times  \Big(p_{s_r-s_{r-1}}(y_{r-1},y_r) - p_{\eta_n(s_r)-\eta_n(s_{r-1})}(y_{r-1},y_r)\Big) \left(\prod_{i=r}^k p_{s_{i+1}-s_{i}}(y_{i},y_{i+1})\right).
\end{align*}

Let us estimate the $r$-th term in the last line in (\ref{sumJi}). We have
\begin{align*}
\int_{S_{k,0,T}}&\int_{(\mathbb{R}^d)^{k+1}} \left(\prod_{i=1}^kh(y_i)\right)f(z) J_{s_1, \dots, s_k,T}^{(r)}(x,y_1, \dots, y_k,z) dz \prod_{j=1}^k dy_j \prod_{i=1}^k ds_{i}
\\&
\leq \|h\|^k \|f\| \int_{S_{k,0,T}}\int_{(\mathbb{R}^d)^{k+1}} \ \big| J_{s_1, \dots, s_k,T}^{(r)}(x,y_1, \dots, y_k,z) \big| dz \prod_{j=1}^k dy_j \prod_{i=1}^k ds_{i}.
\end{align*}
Since the case  $\beta=1$ was already treated in \cite{kul-gan2}, for the rest of the proof we assume that
$\beta>1$.

Consider  two cases: a) $s_r-s_{r-1}> k_{n,\beta}T/n$ and b) $s_r-s_{r-1}\leq k_{n,\beta}T/n$.

In case a), using  condition \textbf{X} and the Chapman-Kolmogorov equation, we derive
$$\ba
&\int_{(\mathbb{R}^d)^{k+1}} \big| J_{s_1, \dots, s_k,T}^{(r)}(x,y_1, \dots, y_k,z) \big| dz \prod_{j=1}^k dy_j \\&\leq B_{T,X} \int_{(\mathbb{R}^d)^2}p_{\eta_n(s_{r-1})-\eta_n(s_0)}(x,y_{r-1})
 \left|\int_{\eta_n(s_r)-\eta_n(s_{r-1})}^{s_r-s_{r-1}} v^{-\beta} q_{v,y_{r-1}}(y_r)dv\right|dy_{r-1}dy_r.
 \ea$$
 Since $k_{n,\beta}\geq 2$, then in  case a) we have $s_r-s_{r-1}\geq 2T/n$, and hence
$$
\eta_n(s_r)-\eta_n(s_{r-1})\geq s_r-\frac{T}{n}-s_{r-1}\geq \frac{s_r-s_{r-1}}{2}.
$$
Therefore, using the fact that $q_{t,y}(\cdot)$ is the probability density for any $t>0$ and $y\in \mathbb{R}^d$, we finally get
\begin{align*}
\int_{(\mathbb{R}^d)^{k+1}} \big| J_{s_1, \dots, s_k,T}^{(r)}(x,y_1, \dots, y_k,z) \big| dz \prod_{j=1}^k dy_j&\leq  B_{T,X} \int_{s_r-s_{r-1}-T/n}^{s_r-s_{r-1}} v^{-\beta}dv\leq \frac{B_{T,X}T}{n} \Big(\frac{s_r-s_{r-1}}{2}\Big)^{-\beta}.
\end{align*}
In case b) we simply apply the Chapman-Kolmogorov equation:
 $$\ba
&\int_{(\mathbb{R}^d)^{k+1}} \big| J_{s_1, \dots, s_k,T}^{(r)}(x,y_1, \dots, y_k,z) \big| dz \prod_{j=1}^k dy_j\\
&\leq
\int_{(\mathbb{R}^d)^2}p_{\eta_n(s_{r-1})-\eta_n(s_0)}(x,y_{r-1}) \Big(p_{s_r-s_{r-1}}(y_{r-1},y_r) + p_{\eta_n(s_r)-\eta_n(s_{r-1})}(y_{r-1},y_r)\Big) dy_{r-1}dy_r\\
 &\leq 2.
 \ea$$
Therefore, summarizing the estimates obtained in cases a) and b) we get, using \eqref{beta1}, the estimates \begin{align*}
\int_{S_{k,0,T}}&\int_{(\mathbb{R}^d)^{k+1}} \ \big| J_{s_1, \dots, s_k,T}^{(r)}(x,y_1, \dots, y_k,z) \big| dz \prod_{j=1}^k dy_j \prod_{i=1}^k ds_{i}\\
&\leq \frac{B_{T,X}T}{n}  2^{\beta} \int_0^T \int_0^{s_r-k_{n,\beta}T/n} \frac{s_{r-1}^{r-1}}{(r-1)!}(s_r-s_{r-1})^{-\beta}\frac{(T-s_r)^{k-r}}{(k-r)!}ds_{r-1}ds_r\\
&+ 2\int_0^T\int_{s_r-k_{n,\beta}T/n}^{ s_r}\frac{s_{r-1}^{r-1}}{(r-1)!}\frac{(T-s_r)^{k-r}}{(k-r)!} ds_{r-1}ds_r\\
&\leq \frac{B_{T,X}T}{n}  2^{\beta} \int_0^T \frac{s_{r}^{r-1}}{(r-1)!}\frac{(T-s_r)^{k-r}}{(k-r)!}ds_r\Big(\int_{k_{n,\beta}T/n}^{T}u^{-\beta}du\Big)\\
&+ \frac{2Tk_{n,\beta}}{n}\int_0^T\frac{s_{r}^{r-1}}{(r-1)!}\frac{(T-s_r)^{k-r}}{(k-r)!} ds_r\\
&\leq \frac{ 2^{\beta} B_{T,X}T^{k+2-\beta}}{(\beta-1)(k-1)!}  n^{-1/\beta}+ \frac{4T^{k+1}}{(k-1)!} n^{-1/\beta}\\
&\leq \frac{2^{\beta\vee 2} T^{k+1}B_{T,X} D_{T,\beta}(n)}{(k-1)!},
\end{align*}
where in the fourth and the fifth lines we used that  $s_{r-1}^{r-1}\leq s_r^{r-1}$.
Taking into account that in \eqref{sumJi} we have $k$ terms, and the common  multiplier $k!$, we finally arrive at \eqref{t2-eq}.
\end{proof}

\section{Condition \textbf{X} for solutions to L\'evy driven SDEs}\label{s3}

Consider  the  SDE
\begin{equation}\label{sde1}
dX_t = b(X_t) dt + dZ_t, \quad X_0=x,
\end{equation}
where $Z$ is a real-valued   L\'evy process. In  \cite{KK14} it was shown that if $Z_t$ is a symmetric $\alpha$-stable process and  $b(\cdot)$ is bounded and Lipschitz continuous, then the  solution to equation  \eqref{sde1} satisfies condition \textbf{X} with $\beta=\max(1, 1/\alpha)$ (in fact, in \cite{KK14} more general multidimensional SDEs of the form \eqref{SDE} are considered). In this section we outline the argument which makes it possible to extend the class of ``L\'evy noises''. Namely, we will omit the requirement on $Z$ to be \emph{symmetric}, and relax the stability assumption, demanding $Z$ to be ``locally $\alpha$-stable'' in the sense we specify below.

Recall that for a real-valued L\'evy process the characteristic function  is of the form
$$
\mathbb{E} e^{i\xi Z_t}= e^{-t\psi(\xi)}, \quad t>0, \, \xi \in \real,
$$
where the \emph{characteristic exponent} $\psi$ admits the L\'evy-Khinchin representation
\begin{equation}\label{psi1}
\psi(\xi) =-ia\xi+{1\over 2}\sigma^2\xi^2+\int_\real \big(1-e^{i \xi u}+ i \xi u\I_{\{|u|\leq 1\}} \big)\mu(du).
\end{equation}
In what follows, we assume that $\sigma^2=0$ and the \emph{L\'evy measure} $\mu$ is of the form  \begin{equation}\label{tilm}
 \mu(du)=C_{+} u^{-1-\alpha}\I_{u\in (0,1)}du +C_{-} |u|^{-1-\alpha}\I_{u\in (-1,0)}du+  m(u) du,
 \end{equation}
 with some $C_\pm\geq 0$ and $m(u)\geq 0$ such that $m(u)=0$ for $|u|\leq 1$, and
 \begin{equation}\label{m1}
 m(u)\leq c|u|^{-1-\alpha}, \quad |u|\geq 1.
\end{equation}
On the interval $[-1,1]$ the L\'evy measure $\mu$ given by (\ref{tilm}) coincides with the L\'evy measure of a (non-symmetric) $\alpha$-stable process. This is the reason for us to call $Z$ a ``locally $\alpha$-stable'' process: its ``local behavior'' near the origin is similar to those of the $\alpha$-stable process. In that context condition (\ref{m1}) means that the ``tails'' of the L\'evy measure for $\mu$ are dominated by the ``tails'' the $\alpha$-stable L\'evy measure.

Let us impose three minor conventions, which will simplify the technicalities below. First, since we are mostly interested in the case $\beta>1$, we assume  that $\alpha<1$. Second, the latter assumption assures that the integral
$$
\int_{\{|u|\leq 1\}}  u\mu(du)
$$
is well defined, and we assume that the constant $a$ in \eqref{psi1} equals to this integral; that is, $\psi$ has the form
$$
\psi(\xi) =\int_\real \big(1-e^{i \xi u}\big)\mu(du).
$$
Clearly, this does not restrict the generality because one can change the constant $a$ by changing respectively the drift coefficient $b(\cdot)$ in \eqref{sde1}. Finally, in order to avoid the usage of the Rademacher theorem (see  \cite[Lemma~7.4]{KK14} for the case when $b$ is just Lipschitz continuous), let us assume that $b\in C^1(\real)$.

In what follows we show how the \emph{parametrix construction} developed in \cite{KK14} can be modified to provide the representation and the bonds for the transition probability density $p_t(x,y)$ of the solution to \eqref{sde1} driven by the ``locally stable'' noise $Z$.

 Let us introduce some notation and give some preliminaries. We denote the space and the space-time convolutions respectively by
$$
(f\ast g)(x,y):=\int_{\Re^d}f(x,z)g(z,y)\, dz,
$$
$$
(f\star g)_t(x,y):=\int_0^t(f_{t-s}\ast g_s)(x,y)\, ds=\int_0^t\int_{\Re^d}f_{t-s}(x,z)g_{s}(z,y)\, dzds.
$$

Generically, the parametrix construction provides the representation of the required transition probability density in the form
\begin{equation}\label{p10}
p_t(x,y)= p_t^0(x,y)+ \int_0^t \int_\real p^0_{t-s}(x,z) \Psi_s (z,y) dzds, \quad t>0, \quad x,y\in \real.
\end{equation}
Here $p^0_t(x,y)$ is a  ``zero approximation term'' for the unknown $p_t(x,y)$,  the function $\Psi_t(x,y)$ is given by the  ``convolution series''
\begin{equation}\label{Psi}
\Psi_t(x,y)= \sum_{k=1}^\infty \Phi_t^{\star k} (x,y), \quad t>0, \quad x,y\in \real,
\end{equation}
 the function $\Phi_t(x,y)$ depends on the particular choice  of $p^0_t(x,y)$, and equals
\begin{equation}\label{phi10}
\Phi_t(x,y):= \big(L_x - \partial_t\big)p_t^0(x,y),
\end{equation}
where
\begin{align*}
L f(x):&=  b(x) f'(x)+ \int_\real \big(f(x+u)-f(x)\big) \mu(du), \quad f\in C^2_b (\real)
\end{align*}
is the formal generator of the process $X$. The subscript $x$  in above expressions means that the operator is applied with respect to  the variable $x$. Note that
to make the above construction feasible, one should properly choose the ``zero approximation term'' $p_t^0(x,y)$, so that the convolution series \eqref{Psi} converges and the space-time convolution in  (\ref{p10}) is well defined. To introduce  in our setting such $p_t^0(x,y)$, and then to construct the bounds for the associated $\Phi_t(x,y)$ and its convolution powers,  we need some more notation.

Denote by $Z^{(\alpha, C_\pm)}$ the $\alpha$-stable process with the L\'evy measure $\mu_{\alpha, C_\pm}(du)=m^{(\alpha, C_\pm)}(u)\, du$, $$
m^{(\alpha, C_\pm)}(u):=C_{+} u^{-1-\alpha}\I_{u>0}du +C_{-} (-u)^{-1-\alpha}\I_{u<0},
$$
and the characteristic exponent
$$
\psi^{(\alpha, C_\pm)}(\xi)= \int_\real \big(1-e^{i\xi u}\big)\mu^{(\alpha, C_\pm)}(du).
$$
Note that since
$$
\psi^{(\alpha, C_\pm)}(c\xi)=c^\alpha \psi^{(\alpha, C_\pm)}(\xi),\quad c>0,
$$
the process $Z^{(\alpha, C_\pm)}$ possesses the scaling property
$$
\mathrm{Law}\,\big(Z_{ct}^{(\alpha, C_\pm)}\big)=\mathrm{Law}\,\big(c^{1/\alpha}Z_t^{(\alpha, C_\pm)}\big), \quad c>0.
$$
Denote by $g_t^{(\alpha,C_\pm)} $ the distribution  density of $Z^{(\alpha, C_\pm)}_t$. By the scaling property we have
$$
g_t^{(\alpha,C_\pm)}(x)=t^{-1/\alpha}g^{(\alpha,C_\pm)}\left(xt^{-1/\alpha}\right), \quad g^{(\alpha,C_\pm)}:=g^{(\alpha,C_\pm)}_1.
$$
Denote  also  by $Z^{(\alpha)}$ the \emph{symmetric} $\alpha$-stable process; that is,  the process of the form introduced above with $C_+=C_-=1$. Let    $g_t^{(\alpha)} $ be the respective distribution density and $g^{(\alpha)}:= g_1^{(\alpha)}$.

 Finally, denote by $\chi_t(x)$ and $\theta_t(y)$, respectively,  the solutions to the ODEs
\begin{equation}\label{ODE}
d\chi_t = b(\chi_t)dt, \quad \chi_0=x,\quad d\theta_t =- b(\theta_t)dt, \quad \theta_0=y.
\end{equation}
Note that these solutions exist, because  $b(\cdot)$ is Lipschitz continuous.

Now we are ready to formulate the main statement of this section.

 \begin{thm}\label{lem-der} Let
\begin{equation}\label{p0}
p^0_t(x,y):=g^{(\alpha,C_\pm)}_t(\theta_t(y)-x).
\end{equation}
Then  the  convolution series \eqref{Psi} is well defined, and the formula (\ref{p10}) gives the representation of the transition  probability density $p_t(x,y)$ of the process $X$.  This density and its time derivative have the following upper bounds:
\begin{equation}\label{ptx}
p_t(x,y)\leq C \big(g_{t+1}^{(\alpha)}+ g_t^{(\alpha)}\big)(y-\chi_t(x)),\quad t\in (0,T], \, x,y\in \real,
\end{equation}
\begin{equation}\label{ptx1}
\prt_t p_t(x,y)\leq C t^{-1/\alpha}\big(g_{t+1}^{(\alpha)}+ g_t^{(\alpha)}\big)(y-\chi_t(x)), \quad t\in (0,T], \, x,y\in \real,
\end{equation}
Consequently, the process $X$ satisfies  condition  \textbf{X} with $\beta= 1/\alpha$ and
$$
q_{t,x}(y)=\big(g_{t+1}^{(\alpha)}+ g_t^{(\alpha)}\big)(y-\chi_t(x)).
$$
 \end{thm}
 \begin{proof} First we evaluate $\Phi_t(x,y)$.    If it is not stated otherwise, we assume in all estimates obtained below that $t\in (0,T]$ for some  $T>0$, and $x,y\in \real$.  Observe that $g^{(\alpha,C_\pm)}\in C_b^2(\real)$. Indeed, this  property  easily follows from the Fourier inversion formula and the expression for the characteristic function. It is known that  $g_t^{(\alpha,C_\pm)}(y-x)$ is the fundamental solution to $\prt_t -L^{(\alpha,C_\pm)}$, where $L^{(\alpha,C_\pm)}$ denotes the generator of the process $Z^{(\alpha,C_\pm)}$:
\begin{equation}\label{la}
L^{(\alpha,C_\pm)}f(x)=\int_\real \big( f(x+u)-f(x)\big) \mu^{(\alpha,C_\pm)}(du), \quad f\in C^2_b(\real).
\end{equation}
Since
$$
(\prt_t-L_x^{(\alpha,C_\pm)})g_t^{(\alpha,C_\pm)}(y-x)=0,
$$
we have
$$\ba
\prt_tp_t^0(x,y)&=\left[\prt_tg_t^{(\alpha,C_\pm)}(w)+\frac{\prt_t\theta_t(y)}{t^{2/\alpha} } (g^{(\alpha,C_\pm)})'\left(\frac{w}{t^{1/\alpha}} \right)\right]\Big|_{w=\theta_t(y)-x}
\\&=\left[\frac{1}{t^{1/\alpha} } (L^{(\alpha,C_\pm)} g^{(\alpha,C_\pm)})\left(\frac{w}{t^{1/\alpha}} \right)+\frac{\prt_t\theta_t(y)}{t^{2/\alpha} } (g^{(\alpha,C_\pm)})'\left(\frac{w}{t^{1/\alpha}} \right)\right]\Big|_{w=\theta_t(y)-x},
\ea
$$
where in the last identity we  used the scaling property of  $g_t^{(\alpha,C_\pm)}$ and the fact that $L^{(\alpha,C_\pm)}$ is a homogeneous operator of  order $1/\alpha$. Next, by the very definition of  $L $ and $ L^{(\alpha,C_\pm)}$ we get
\begin{equation*}
\begin{split}
L_x&p_t^0(x,y)=\left[\frac{1}{t^{1/\alpha} } (L^{(\alpha,C_\pm)} g^{(\alpha)})\left(\frac{w}{t^{1/\alpha}} \right)-\frac{b(x)}{t^{2/\alpha} } (g^{(\alpha,C_\pm)})'\left(\frac{w}{t^{1/\alpha}} \right)\right]\Big|_{w=\theta_t(y)-x}\\
&+ \left[ \int_{|u|\geq 1}
\left(\frac{1}{t^{1/\alpha} } g^{(\alpha,C_\pm)}\left(\frac{w-u}{t^{1/\alpha}} \right)-
\frac{1}{t^{1/\alpha} } g^{(\alpha,C_\pm)}\left(\frac{w}{t^{1/\alpha}} \right)\right)\big(m(u)du-m^{(\alpha,C_\pm)}(du)\big)\right]\Big|_{w=\theta_t(y)-x}.
\end{split}
\end{equation*}
Therefore, using the relation  $\prt_t\theta_t(y)=-b(\theta_t(y))$, we get
\begin{equation}\label{phi-n10}
\begin{split}
\Phi_t(x,y)&=\big(L_x - \partial_t\big)p_t^0(x,y)=\frac{b(x)-b(\theta_t(y))}{t^{2/\alpha} } (g^{(\alpha,C_\pm)})'\left(\frac{\theta_t(y)-x}{t^{1/\alpha}} \right)\\
&+\frac{1}{t^{1/\alpha} }\int_{|u|\geq 1}
\left( g^{(\alpha,C_\pm)}\left(\frac{\theta_t(y)-x-u}{t^{1/\alpha}} \right)-
 g^{(\alpha,C_\pm)}\left(\frac{\theta_t(y)-x}{t^{1/\alpha}} \right)\right)\big(m(u)-m^{(\alpha,C_\pm)}(u)\big)\, du\\
&=: \Phi_t^1(x,y)+\Phi_t^2(x,y).
\end{split}
\end{equation}
Further, we  give the  bounds for the absolute values of $\Phi^1_t(x,y)$, $\Phi^2_t(x,y)$, and $\Phi_t(x,y)$. In what follows, $C$ denotes a generic constant, whose value might  be  different from place to place. One has
\begin{equation}\label{ga9}
g^{(\alpha,C_\pm)}(x)\leq C g^{(\alpha)}(x), \quad \quad x\in \real,
\end{equation}
\begin{equation}\label{ga91}
\big|(g^{(\alpha,C_\pm)})'(x)\big|\leq C (1+|x|)^{-1} g^{(\alpha)}(x),\quad x\in \real,
\end{equation}
\begin{equation}\label{ga92}
\big|(g^{(\alpha,C_\pm)})''(x)\big|\leq C (1+|x|)^{-2} g^{(\alpha)}(x),\quad x\in \real.
\end{equation}
Since the argument used in the proof of  (\ref{ga9}) -- (\ref{ga92}) is quite standard (see e.g.  \cite[Appendix A]{KK14}), we omit the details.

By \eqref{ga91} and the Lipschitz continuity of $b(\cdot)$ we have
$$
|\Phi_t^1(x,y)|\leq\frac{C|x-\theta_t(y)|}{t^{2/\alpha} }\left| (g^{(\alpha,C_\pm)})'\left(\frac{\theta_t(y)-x}{t^{1/\alpha}} \right)\right|\leq \frac{C}{t^{1/\alpha} } g^{(\alpha)}\left(\frac{\theta_t(y)-x}{t^{1/\alpha}} \right)=Cg_t^{(\alpha)}(\theta_t(y)-x).
$$

To get the estimate for $|\Phi_t^2(x,y)|$, we first observe that
$$
\big|m(u)-m^{(\alpha,C_\pm)}(u)\big|I_{|u|\geq 1} \leq C g^{(\alpha)}(u),
$$
which implies
$$\ba
|\Phi_t^2(x,y)|&\leq \frac{C}{t^{1/\alpha} }\int_{|u|\geq 1}
g^{(\alpha,C_\pm)}\left(\frac{\theta_t(y)-x-u}{t^{1/\alpha}} \right)
 g^{(\alpha)}(u)\, du+\frac{C}{t^{1/\alpha} }g^{(\alpha,C_\pm)}\left(\frac{\theta_t(y)-x}{t^{1/\alpha}} \right).
\ea
$$
Taking into account \eqref{ga9}, we deduce that
$$
|\Phi_t^2(x,y)|\leq C \big( g_t^{(\alpha)}* g_1^{(\alpha)}+ g_t^{(\alpha)}\big)(\theta_t(y)-x)= C \big( g_{t+1}^{(\alpha)}(x)+ g_t^{(\alpha)}(x)\big)(\theta_t(y)-x).
$$
Combining the estimates for $\Phi_t^1(x,y)$ and $\Phi_t^2(x,y)$, we get
\begin{equation}\label{phi-n30}
|\Phi_t(x,y)| \leq C  \big(g_{t+1}^{(\alpha)}+ g_t^{(\alpha)}\big)(\theta_t(y)-x).
\end{equation}

Our next step is to estimate the convolution powers of $\Phi$.
It is shown in \cite[Appendix B]{KK14}, that the kernel $g_t^{(\alpha)}(\theta_t(y)-x)$ possess the following \emph{sub-convolution property}:
\begin{equation}\label{sub10}
\int_\real g_{t-s}^{(\alpha)}\big(\theta_{t-s}(z)-x\big)g_s^{(\alpha)}\big(\theta_s(y)-z\big)dz \leq C g_t^{(\alpha)}(\theta_t(y)-x).
\end{equation}
By this property   we get
\begin{equation}\label{sub20}
\begin{split}
\int_\real g_{t+1-s}^{(\alpha)}\big(\theta_{t-s}(z)-x\big)g_s^{(\alpha)}\big(\theta_s(y)-z\big)dz &\leq C g_{t+1}^{(\alpha)}(\theta_t(y)-x), \\
\int_{\real} g_{t-s+1}^{(\alpha)}(\theta_{t-s}(z)-x)g_{s+1}^{(\alpha)}(\theta_s(y)-z)dz &\leq C g_{t+2}^{(\alpha)}(\theta_t(y)-x)\leq C g_{t+1}^{(\alpha)}(\theta_t(y)-x),
\end{split}
\end{equation}
where in the last line we used that $g_2^{(\alpha)}\leq C g_1$, and therefore $g_{t+2}^{(\alpha)}=g_{t}^{(\alpha)}*g_{2}^{(\alpha)}\leq Cg_{t+2}^{(\alpha)}$.
Having these estimates, we deduce in the same way as in \cite[Section 3]{KK14} that
\begin{equation}\label{phik}
|\Phi_t^{\star k} (x,y)|\leq\frac{ C_0(C t)^{k-1}}{k!} \big(g_{t+1}^{(\alpha)}+ g_t^{(\alpha)}\big)(\theta_t(y)-x).
\end{equation}
Therefore, the   series \eqref{Psi} converges absolutely for $(t,x,y)\in (0,\infty)\times \real\times \real$,  and
$$
|\Psi_t(x,y)|\leq C \big(g_{t+1}^{(\alpha)}+ g_t^{(\alpha)}\big)(\theta_t(y)-x).
$$
Applying once again the sub-convolution property \eqref{sub10}, we see that the convolution $p^0\star \Psi$ is well defined, and
$$
|\big(p^0\star \Psi\big)_t(x,y)|\leq C  \big(g_{t+1}^{(\alpha)}+ g_t^{(\alpha)}\big)(\theta_t(y)-x).
$$
Thus, the expression \eqref{p10} for $p_t(x,y)$  is well defined for any $(t,x,y)\in (0,\infty)\times \real\times \real$, and
$$
|p_t(x,y)|\leq C  \big(g_{t+1}^{(\alpha)}+ g_t^{(\alpha)}\big)(\theta_t(y)-x).
$$
Finally, to get \eqref{ptx} we use the following inequalities, which were proved in \cite[Appendix B]{KK14}:
\be\label{flow}
c |\theta_t(y)-x|\leq |\chi_t(x)-y|\leq C |\theta_t(y)-x|.
\ee
Since  for any constant $c>0$ we have $g_t^{(\alpha)}(x)\asymp g_t^{(\alpha)}(c x)$  for any $t\in (0,T]$, $x,y\in \real$, this completes the proof of \eqref{ptx}.

Our final step is to use representation \eqref{p10} in order to find the bounds for $\partial_t p_t(x,y)$. Since $p_t^0(x,y)$ and $\Phi_t(x,y)$ are given explicitly, it is  straightforward  to show that these functions are differentiable with respect to  $t$, and
to check using \eqref{ga9} -- \eqref{ga92} that
\begin{equation}\label{p20}
\big| \prt_t p_t^0(x,y)|\leq C t^{-1/\alpha} g_t^{(\alpha)}(\theta_t(y)-x),
\end{equation}
 \begin{equation}\label{tphi}
 |\prt_t \Phi_t(x,y)|\leq C   t^{-1/\alpha}\big( g_{t+1}^{(\alpha)}+ g_t^{(\alpha)}\big)(\theta_t(y)-x).
 \end{equation}

 To show that the convolution powers $\Phi_t^{\star k}(x,y)$ are differentiable in $t$ and to get the upper bounds, we use the following trick. The expression  for $\Phi_t^{\star (k+1)}(x,y)$ can be re-organized as follows:
\begin{equation}\label{44}
\begin{split}
\Phi^{\star (k+1)}_t(x,y)&=\int_0^{t}\int_{\real}\Phi_{t-s}^{\star k}(x,z)\Phi_s(z,y)\,dzds\\
&=\int_0^{t/2}\int_{\real}\Phi_{t-s}^{\star k}(x,z)\Phi_s(z,y)\,dz ds+\int_0^{t/2}\int_{\real}\Phi_{s}^{\star k}(x,z)\Phi_{t-s}(z,y)\,dz ds.
\end{split}
\end{equation}
If  $k=1$, the first line in \eqref{44} does not allow us to differentiate  $\prt_t\Phi^{\star (2)}_t(x,y)$, because the upper bound for  $\prt_t\Phi_{t-s}(x,z)$ has a  non-integrable singularity $(t-s)^{-1/\alpha}$ at the vicinity of the point $s=t$ (recall that $\alpha<1$). However, the identity given by the second line in \eqref{44} does not contain such singularities, and we can  show using  induction that
 for any $k\geq 1$ the function  $\Phi^{\star k}_t(x,y)$ is continuously differentiable  in  $t$, satisfies
\begin{align*}
\prt_t\Phi^{\star (k+1)}_t(x,y)&=\int_0^{t/2}\int_{\Re^d}(\prt_t\Phi^{\star k})_{t-s}(x,z)\Phi_s(z,y)\,dz ds+
\int_0^{t/2}\int_{\Re^d}\Phi_{s}^{\star k}(x,z)(\prt_t\Phi)_{t-s}(z,y)\,dz ds\\
&\quad +\int_{\Re^d}\Phi_{t/2}^{\star k}(x,z)\Phi_{t/2}(z,y)\, dz.
\end{align*}
and possesses the bound
\begin{equation}\label{phikd}
|\partial_t\Phi_t^{\star k} (x,y)|\leq \frac{ C_0(Ct)^{k-1} t^{-1/\alpha}}{k!} \big(g_{t+1}^{(\alpha)}+ g_t^{(\alpha)}\big) (\theta_t(y)-x), \quad\quad k\geq 1.
\end{equation}
Since the proof is completely analogous to the proof of  \cite[Lemma~7.3]{KK14},  we omit the details.

From \eqref{phikd} we derive the following  bound for the derivative of $\Psi_t(x,y)$:
\begin{equation}\label{phikd2}
|\partial_t\Psi_t^{\star k} (x,y)|\leq C t^{-1/\alpha}  \big(g_{t+1}^{(\alpha)}+ g_t^{(\alpha)}\big) (\theta_t(y)-x).
\end{equation}
Re-organizing representation (\ref{p10}) in the same way as  (\ref{44}), we get
 $$
p_t(x,y)
=p_t^0(x,y)+\int_0^{t/2}\int_{\rdd}p_{t-s}^0(x,z)\Psi_s(z,y)\,dzds+\int_{0}^{t/2}\int_{\rdd}p_{s}^0(x,z)\Psi_{t-s}(z,y)\,dz\, ds.
$$
Using the above representation of $p_t(x,y)$ together with  \eqref{p20} and  \eqref{phikd2}, we derive the existence of the continuous  derivative $\prt_t p_t(x,y)$, which satisfies the  inequality
$$
|\prt_t p_t(x,y)|\leq C   t^{-1/\alpha}\big( g_{t+1}^{(\alpha)}+ g_t^{(\alpha)}\big)(\theta_t(y)-x).
$$
Using estimates (\ref{flow}) in the same way as  we did that in the proof of (\ref{ptx}), we can change the argument $\theta_t(y)-x$ in the right hand side of the above estimate  to $y-\chi_t(x)$, which completes the proof of (\ref{ptx1}).
 \end{proof}

\section{Application: the price of an occupation time option}\label{s4}

In this section, we consider an \emph{occupation time option} (see \cite{Linetsky}),  with the price of the option depending on the time spent by  an asset price process  in a  given set.  Comparing to the standard barrier options, which are activated or cancelled when the asset price process hits some definite  level (barrier), the payoff of the occupation time option depends on the time during which  the asset price process stays below or above such a barrier.

For instance, for the strike price $K$, the barrier level $L$ and the knock-out rate $\rho$, the payoff of a down-and-out call occupation time option equals
$$
\exp \left( - \rho \int_0^T \mathbb{I}_{\{S_t \leq L\}} dt \right) (S_T - K)_+,
$$
and the price $\mathbf{C}(T)$ of the option is defined as
$$
\textbf{C}(T) = \exp (-rT) E \left[ \exp \left( - \rho \int_0^T \mathbb{I}_{\{S_t \leq L\}} dt \right) (S_T - K)_+ \right]
$$
where $r$ is the risk-free interest rate (see \cite{Linetsky}).

Assume  that the price of an asset $S=\{S_t, t\geq 0\}$  is of the form
$$
S_t = S_0 \exp(X_t),
$$
 where $X$ is the Markov process  studied in previous sections. Then the time spent by the process $S$ in a set $J\subset \mathbb{R}$ equals to the time spent by $X$ in the set $J'=\log J$.

Let us approximate  the price $\textbf{C}(T)$ of our  option by
$$
\textbf{C}_n(T) = \exp (-rT) E \left[ \exp \left( - \rho T/n \sum_{k=0}^{n-1} \mathbb{I}_{\{S_{kT/n} \leq L\}} dt \right) (S_T - K)_+ \right].
$$
Then using the results from the previous sections we can get the control on  the accuracy of such an approximation.

First we apply Theorem \ref{t1} and derive the strong approximation rate.

 \begin{prop}\label{p41}
Suppose that the process  $X$ satisfies condition  \textbf{X}, and assume that there exists $\lambda>1$ such that $G(\lambda):= E \exp (\lambda X_T) = E (S_T)^\lambda < + \infty$.

Then
$$
\Big|\textbf{C}_n(T)- \textbf{C}(T)\Big| \leq \exp(-rT) \rho G(\lambda)^{1/\lambda} C_{T,\lambda/(\lambda-1)}  (D_{T,\beta} (n))^{1/2},
$$
with  constants $C_{T,\lambda/(\lambda-1)}$ and $D_{T,\beta} (n)$  are given by \eqref{DC}.
\end{prop}
\begin{proof} The proof is a simple corollary of Theorem \ref{t1}. Denote $h(x)=\rho \mathbb{I}_{x\leq \log L}$. Then keeping the notation of Section \ref{s2} we get
$$
\textbf{C}(T)=e^{-rT}Ee^{-I_T(h)}(S_T- K)_+, \quad \textbf{C}_n(T)=e^{-rT}Ee^{-I_{T,n}(h)}(S_T- K)_+.
$$
By the H\"older inequality with $p=\lambda$ and  $q=\lambda/(\lambda-1)$,
$$
\Big|\textbf{C}_n(T)- \textbf{C}(T)\Big| \leq e^{-rT} \Big(E(S_T- K)_+^\lambda\Big)^{1/\lambda}\left(E\left|e^{-I_T(h)}-e^{-I_{T,n}(h)}\right|^{\lambda/(\lambda-1)}\right)^{(\lambda-1)/\lambda}.
$$
Since for positive $a$ and $b$ we have  $|e^{-a}-e^{-b}|\leq |a-b|$, then
$$
\Big|\textbf{C}_n(T)- \textbf{C}(T)\Big| \leq e^{-rT} G(\lambda)^{1/\lambda} \left(E\left|I_T(h)-I_{T,n}(h)\right|^{\lambda/(\lambda-1)}\right)^{(\lambda-1)/\lambda},
$$
and thus the required statement follows directly from Theorem~\ref{t1} with $p=\lambda/(\lambda-1)$.
\end{proof}

We  also can  control the accuracy of the approximation using the weak rate bound from Theorem \ref{t2}. Observe that the bound given below is  sharper  than those  obtained in the previous proposition precisely  when $\lambda>2$.

\begin{prop} Under the assumptions of Proposition~\ref{p41}, we have
$$
\Big|\textbf{C}_n(T)- \textbf{C}(T)\Big| \leq 2^{\beta\vee 2+1}\max\{B_{T,X} \rho T^2 (1+\rho T) \exp(\rho T), G(\lambda) \} \exp (-rT)  \widetilde{D}_{T, \beta} (n),
$$
where
$$
\widetilde{D}_{T,\beta} (n) =
\begin{cases}
n^{-(1-1/\lambda)} \log n,& \beta=1,\\
\max\left(1,\frac{T^{1- \beta}}{\beta-1}\right) n^{-1/ \beta (1-1/\lambda)},& \beta>1.
\end{cases}
$$

\end{prop}

\begin{proof}

For some $N>0$ denote
$$\ba
&\textbf{C}^N(T) = \exp (-rT) E \left[ \exp \left( - \rho \int_0^T \mathbb{I}_{\{S_t \leq L\}} dt \right)  ((S_T - K)_+ \wedge N) \right],\\&
\textbf{C}^N_n(T) = \exp (-rT) E \left[ \exp \left( - \rho T/n \sum_{k=0}^{n-1} \mathbb{I}_{\{S_{kT/n} \leq L\}} dt \right)  ((S_T - K)_+ \wedge N) \right].
\ea$$

Then
\be
\label{exmpl_bound}
\Big|\textbf{C}_n(T)- \textbf{C}(T)\Big| \leq \Big|\textbf{C}^N_n(T)- \textbf{C}^N(T)\Big| + \Big|\textbf{C}(T)- \textbf{C}^N(T)\Big| + \Big|\textbf{C}_n(T)- \textbf{C}^N_n(T)\Big|.
\ee

We estimate each term separately.

 Using  Theorem \ref{t2} and the Taylor expansion of the exponent, we derive
\begin{align*}
\Big|\textbf{C}^N_n(T)- \textbf{C}^N(T)\Big| &\leq 2^{\beta\vee 2}B_{T,X} NT\exp (-rT) \sum_{k=1}^{\infty} \frac{\rho^k}{k!} k^2 T^{k}  D_{T,\beta} (n)\\
&= 2^{\beta\vee 2}B_{T,X} \rho N T^2 (1+\rho T) \exp(\rho T-rT) D_{T,\beta} (n).
\end{align*}
For the last two terms we get
\begin{align*}
\Big|\textbf{C}(T)- \textbf{C}^N(T)\Big| + \Big|\textbf{C}_n(T)&- \textbf{C}^N_n(T)\Big| \leq 2 \exp (-rT) E\left[(S_T - K)_+ -(S_T - K)_+ \wedge N \right]
\\
& \leq 2 \exp (-rT) E [S_T \mathbb{I}_{\{S_T>N\}}] \\
&= 2 \exp (-rT) E\left[ \frac{S_T N^{\lambda-1}\mathbb{I}_{\{S_T>N\}}}{N^{\lambda-1}} \right] \leq \frac{2G(\lambda)}{N^{\lambda-1}} \exp (-rT).
\end{align*}
To complete the proof, put $N = n^{1/(\beta \lambda)}$.

\end{proof}


\begin{thebibliography}{99}


\bibitem{BGR14} K.\ Bogdan, T.\  Grzywny, M.\ Ryznar.  Density and tails of unimodal convolution semigroups.  \emph{J. Func. An.} \textbf{ 266(6)} (2014).  3543--3571.

\bibitem{Dynkin}
  E.\ B.\ Dynkin.   \emph{Markov processes}. Fizmatgiz, Moscow, 1963 (in Russian).

\bibitem{kul-gan}  Iu.\ Ganychenko, A. Kulik,   \emph{
Rates of approximation of nonsmooth integral-type functionals of Markov processes},
Modern Stochastics: Theory and Applications, \textbf{2} (2014), 117--126.


\bibitem{Gobet}
E.\ Gobet, C.\ Labart.
Sharp estimates for the convergence of the density of the Euler scheme in small time,
\emph{Elect. Comm. in Probab.} \textbf{13} (2008), 352--363.

\bibitem{Guerin}
 H.\ Guerin, J.-F. \ Renaud.
Joint distribution of a spectrally negative Levy process and its occupation time, with step option pricing in view.  Available at \href{http://arxiv.org/pdf/1406.3130v1.pdf}{http://arxiv.org/pdf/1406.3130v1.pdf}



\bibitem{KS13a} K.\ Kaleta, P.\ Sztonyk.  Small time sharp bounds for kernels of convolution semigroups. To appear in \emph{J. Anal. Math. }

\bibitem{KS13b} K.\ Kaleta, P.\ Sztonyk. Estimates of transition densities and their derivatives for jump L\'evy processes. To appear in \emph{J. Math. Anal. Appl.}

\bibitem{KK12a}  V.\ Knopova, A.\ Kulik. Intrinsic small time estimates for distribution densities of L\'evy processes.  \emph{Random Op. Stoch. Eq.} \textbf{21(4)}  (2013), 321--344.


\bibitem{K13} V. Knopova.  Compound kernel estimates for the transition probability density of a L\'evy process in $\mathbb{R}^n$.   \emph{Th. Probab. Math. Stat.} \textbf{89} (2014), 57--70.


\bibitem{KK14}
V.\ Knopova, A.\ Kulik. Parametrix construction of the transition probability density of the solution to an SDE driven by $\alpha$-stable noise. Preprint 2014.  Available at \href{http://arxiv.org/pdf/1412.8732v1.pdf}{http://arxiv.org/pdf/1412.8732v1.pdf}


\bibitem{Kohatsu-Higa}
A. Kohatsu-Higa, A. Makhlouf, H.L. Ngo.
Approximations of non-smooth integral type functionals of one dimensional diffusion precesses. \emph{Stoch. Proc.  Appl.} \textbf{124(5)} (2014),  1881--1909.

\bibitem{KR15}  T.\ Kulczycki, M.\ Ryznar, Gradient estimates of harmonic functions and transition densities for L\'evy
processes. To appeat in \emph{ Trans. Amer. Math. Soc.}

\bibitem{kul-gan2}
 A.\ Kulik, Iu.\ Ganychenko.
On accuracy of weak approximations of integral-type functionals of Markov processes.  Preprint 2015. (In Ukrainian)

\bibitem{Linetsky}
V.\ Linetsky.
Step options. Math. Finance, \textbf{9(1)} (1999), 55--96.


\bibitem{M12} A.\ Mimica. Heat kernel upper estimates for symmetric jump processes with small jumps
of high intensity.  \emph{Poten. Anal.} \textbf{ 36(2)} (2012), 203--222.


\bibitem{PT69} W.E. Pruitt, S.J. Taylor, The potential kernel and hitting probabilities for the general stable process in $\mathbb{R}^n$.
\emph{Trans. Am. Math. Soc.} \textbf{ 146} (1969), 299--321.

\bibitem{RS10} J. Rosinski and  J.Singlair,  Generalized tempered stable processes. In: \emph{Stability in Probability, Ed. J.K. Misiewicz}, 90, Banach Center Publ.  (2010), 153--170.



\bibitem{St10a} P. Sztonyk. Estimates of tempered stable densities. \emph{J. Theor. Probab.} \textbf{23(1)}  (2010)
 127--147.

\bibitem{St10b} P. Sztonyk, Transition density estimates for jump L\'evy processes. \emph{Stoch. Proc. Appl.} \textbf{121}  (2011), 1245--1265.


\bibitem{Sznitman}
A. Sznitman. \emph{ Brownian motion, obstacles and random media}. Springer, Berlin, 1998.


\bibitem{W07} T. Watanabe, Asymptotic estimates of multi-dimensional stable densities and their applications. \emph{ Trans.
Am. Math. Soc.} \textbf{359(6)} (2007), 2851--2879.

\end{thebibliography}
\end{document}